\documentclass[10pt]{article}
\usepackage{amsmath,amssymb,amsbsy,amsfonts,amsthm,latexsym,
            amsopn,amstext,amsxtra,euscript,amscd,amsthm}

\newtheorem{teor}{Theorem}
\newtheorem{theorem}{Theorem}
\newtheorem{lemma}{Lemma}

\def\A{\alpha}
\def\B{\beta}
\def\E{\eta}
\def\d{\delta}
\def\l{\lambda}

\def\G{\Gamma}

\def\r{\rho}

\def\Z{\mathbb{Z}}

\def\Q{\mathbb{Q}}
\def\R{\mathbb{R}}

\def\L{\mathbb{L}}

\begin{document}

\title{Zeckendorf representations with at most two terms to\\ $x$--coordinates of Pell equations}

\author{
{\sc Carlos Alexis G\'omez Ruiz}\\
{Departamento de Matem\'aticas,~Universidad del Valle,}\\
{Calle 13 No 100--00,~Cali, Colombia}\\
{carlos.a.gomez@correounivalle.edu.co
\vspace*{.3cm}}
\and
{\sc Florian~Luca} \\
{School of Mathematics, University of the Witwatersrand,}\\
{Private Bag X3, Wits 2050, South Africa
\vspace*{.1cm}}
\and
{Max Planck Institute for Mathematics,}\\
{Vivatgasse 7, 53111 Bonn, Germany
\vspace*{.1cm}}
\and
{Department of Mathematics, Faculty of Sciences,}\\
{University of Ostrava, 30 Dubna 22, 701 03 Ostrava 1,~Czech Republic}\\
{florian.luca@wits.ac.za}}

\date{\today}

\pagenumbering{arabic}

\maketitle

\begin{abstract}
\noindent In this paper, we find all positive squarefree integers $d$ such that the Pell equation
$X^2 - dY^2 = \pm1$ has at least two positive integer solutions $(X,Y)$ and $(X',Y')$ such that
both $X$ and $X'$ have Zeckendorf representations with at most two terms.

This paper has been accepted for publication in SCIENCE CHINA Mathematics.
\end{abstract}



\section{Introduction}
\noindent For a positive squarefree integer $d$ and the Pell equation
\begin{equation}
\label{EPell}
X^2-dY^2 = \pm1, \qquad {\text{\rm where}}\qquad X,~Y \in \Z^{+},
\end{equation}
it is well--known that all its solutions $(X,Y)$ have the form
\begin{equation*}
X+Y{\sqrt{d}}=X_k+ Y_k{\sqrt{d}}= (X_1 + Y_1{\sqrt{d}})^k
\end{equation*}
for some $k \in \Z^{+}$, where $(X_{1},Y_{1})$ be the smallest positive integer solution of \eqref{EPell}.
The sequence $\{X_k\}_{k\ge 1}$ is a binary recurrent sequence.  In fact, the formula
\begin{equation}
\label{x_l}
X_k = \frac{(X_1 + \sqrt{d}Y_1)^k+(X_1 - \sqrt{d}Y_1)^k}{2}
\end{equation}
holds for all positive integers $k$.

Recently there was a spur of activity around investigating for which $d$, there are members of sequence $\{X_k\}_{k\ge 1}$ which belong to some interesting sequences of positive integers. Maybe the first result of this kind is due to Ljunggren \cite{Lj}  who showed that if \eqref{EPell} has a solution with $-1$ on the right--hand side, then there is at most one odd $k$ such that $X_k$ is a square. In \cite{Co}, it is shown that if all solutions of \eqref{EPell} have the sign $+1$ on  the right--hand side, then $X_k$ is a square only when $k\in \{1,2\}$, with both $X_1$ and $X_2$ being squares occurring only for $d=1785$. When only solutions with the sign $+1$ in the right--hand side are considered, in \cite{FY} it is shown that $X_k$ is a repdigit in base $10$ for at most one $k$, except when  $d=2$, for which both $X_1=3$ and $X_3=99$ are repdigits, and when $d=3$ for which both $X_1=2$ and $X_2=7$ are repdigits. More generally, in \cite{BF} it is shown that if $b\ge 2$ is any integer, then, under the same assumption that only solutions with the sign $+1$ on the right--hand side are considered, there are only finitely many $d$'s such that $X_k$ is a base $b$-repdigit for at least two values of $k$. All such $d$ are bounded by $\exp((10b)^{10^5})$. In \cite{FT}, it is shown that $X_k$ is a Fibonacci  number for at most one $k$, except for $d=2$ when both $X_1=1$ and $X_2=3$ are Fibonacci numbers.

We recall that the Fibonacci sequence $\{F_k\}_{k\ge0}$ and its companion Lucas sequence $\{L_k\}_{k\ge 0}$ are given by
$F_0 = 0,~~ F_1 = 1$, $L_0=2,~L_1 = 1$ and for both, each term afterwards is the sum of the preceding two terms.

Letting $\A = (1+\sqrt{5})/2$ and $\B = (1-\sqrt{5})/2$ be the roots of the characteristic polynomial $X^2-X-1$ of both the Fibonacci and Lucas sequences, the Binet formulas
\begin{eqnarray}
\label{binet}
F_k = \cfrac{\A^k - \B^k}{\sqrt{5}} \qquad {\rm and} \qquad L_k = \A^k+\B^k
\end{eqnarray}
hold for all nonnegative integers $k$. Further, the inequalities
\begin{eqnarray}
\label{crecFib}
\A^{k-2} \le F_k \le \A^{k-1}\qquad {\text{\rm hold for all}}\qquad  k\ge 1.
\end{eqnarray}

Zeckendorf's theorem (see \cite{Zeck}) claims that every positive integer $N$  has a unique representation as sum of non--consecutive Fibonacci numbers. That is,
$$
N = F_{k_1}+\cdots+F_{k_r},\qquad {\text{\rm where}}\qquad k_{i+1}-k_{i} \ge 2\qquad {\text{\rm for~all}}\qquad i=1,2,\ldots,r-1.
$$
We say that $N$ has Zeckendorf representation with $r$ terms.

In this paper, we look at Pell equations \eqref{EPell} such that $X_\ell$ has Zeckendorf representation with at most two terms, for at least two values of $\ell$.

We prove the following result.

\begin{theorem}
\label{thm:main}
For each squarefree integer $d$, there is at most one positive integer $\ell$ such that  $X_\ell$ has a Zeckendorf representation with at most two terms, except for $d\in \{2,3,5, 11,30\}$.
\end{theorem}

For the exceptional values of $d$ appearing in the statement of Theorem \ref{thm:main}, all solutions $(\ell,m,n)$ of the Diophantine equation
$$
X_{\ell} = F_m + F_n, \qquad {\text{\rm with}}\qquad n-m\ge 2
$$
are listed in the Section $2$ and $5$.  Note that our results also give all solutions of the problem under the more relaxed condition that $n\ge m$ (but not necessarily that $n-m\ge 2$). Namely, if $m=n-1$, we then take $F_n+F_{n-1}=F_{n+1}+F_0$
and when $m=n$ (and $n\ge 2$, since $F_1=F_2$), then $F_n+F_m=2F_n=F_{n+1}+F_{n-2}$.

The main tools used in this work are lower bounds for linear forms in logarithms \'a la Baker and a version of the
Baker--Davenport reduction method from Diophantine approximation, in addition to elementary properties of Fibonacci numbers and solutions to Pell equations.

\section{A preliminary consideration}

First of all, we consider the case $ d = 5 $ in equation \eqref{EPell}. 
It is well--known that if $(X,Y)$ are positive integers such that
$$
X^2-5Y^2=\pm 4,\quad{\rm then}\quad (X,Y)=(L_n,F_n)\quad {\rm for ~~some} ~ ~~ n\ge 1.
$$
In particular, if $X^2-5Y^2=\pm 1$, then $(2X)^2-5(2Y)^2=\pm 4$, so $2X=L_n$ for some integer $n$. Thus, $X=L_n/2$ and since this is an integer, we have $3\mid n$.
One checks (by induction, for example), that
$$
\frac{L_n}{2}=F_n+\frac{F_{n-3}}{2}.
$$
So, assume that $L_n/2=F_a+F_b$ for some $n\ge 10$ and $a\ge b$ (in particular for $a-b \ge2$). If $a\le n-2$, then
$$
F_n<F_n+\frac{F_{n-3}}{2}=\frac{L_n}{2}=F_a+F_b\le 2F_a\le F_{n-2}<F_n,
$$
a contradiction. If $a=n-1$ and $b\le n-2$, then again
$$
F_n<\frac{L_n}{2}=F_a+F_b\le F_{n-1}+F_{n-2}=F_n
$$
a contradiction, while if $a=b=n-1$, then
$$
F_a+F_b=2F_{n-1}=F_n+F_{n-3}>F_n+\frac{F_{n-3}}{2}=\frac{L_n}{2},
$$
again a contradiction. Certainly, if $a\ge n+1$, then
$$
F_a+F_b\ge F_{n+1}=F_n+F_{n-1}>F_n+\frac{F_{n-3}}{2}=\frac{L_n}{2},
$$
again a contradiction. Having explored both possibilities $a\ge n+1$ and $a\le n-1$ without success, we conclude that $a=n$. Hence,
$$
F_a+F_b=F_n+F_b=F_n+\frac{F_{n-3}}{2}=\frac{L_n}{2},
$$
giving $F_b=F_{n-3}/2$. This is also wrong since
$$
F_{n-5}<\frac{F_{n-3}}{2}<F_{n-4}
$$
for $n\ge 10$ meaning $F_{n-3}/2$ cannot be a Fibonacci number $F_b$. Thus, $n<10$, and since $3\mid n$ we need to check $n=3,6,9$. The cases $n=3,6$ give the solutions
\begin{eqnarray*}
F_3 + F_0 = 2F_2 = 2F_1 = F_2+F_1 =2 =X_1,\\
F_6 + F_1 = F_6 + F_2 = 9 = X_2,
\end{eqnarray*}
while the case $n=9$ doesn't since for it we have $X_3=L_9/2=38$ which is not a sum of two Fibonacci numbers.

From now on, we will consider the Pell equation \eqref{EPell} with $ d \neq 5$.

\section{An inequality for $n$ and $\ell$}

\noindent
Let $(X_1, Y_1)$ be the minimal solution in positive integers of the Pell equation \eqref{EPell} with $d\neq 5$.
Taking
\begin{equation}
\label{delta-eta}
\delta :=X_{1}+\sqrt{d}Y_{1} \qquad \text{and} \qquad \eta :=X_{1}-\sqrt{d}Y_{1}.
\end{equation}
We obtain that
\[
\delta\cdot\eta = X_{1}^{2}-dY_{1}^{2}=:\epsilon, \quad \epsilon \in \left\{\pm 1\right\}.
\]
Thus, from \eqref{x_l}, we have
\begin{equation}\label{E2}
X_{\ell}=\frac{1}{2}\left(\delta^{\ell}+\eta^{\ell}\right).
\end{equation}
Since $\delta \geq 1+\sqrt{2}$ and $\alpha = (1+\sqrt{5})/2$, it follows that the estimate
\begin{equation}
\label{E3}
\frac{\delta^{\ell}}{\alpha^2}\leq X_{\ell}<\delta^{\ell} \quad \text{holds for all} \quad \ell \geq 1.
\end{equation}
Indeed, inequality on the right--hand side is taken from the fact that $|\eta|=|\delta|^{-1}$.
To inequality on the left--hand side, we note that
\begin{eqnarray*}
X_{\ell} \ge \frac{\delta^{\ell}-\delta^{-\ell}}{2} = \delta^{\ell}\left(\frac{1-\delta^{-2\ell}}{2}\right) \ge
\delta^{\ell} \left(\frac{1-(1+\sqrt{2})^{-2}}{2}\right) > \cfrac{\delta^{\ell}}{\alpha^{2}}.
\end{eqnarray*}

We assume that $(m_{1},n_{1},\ell_{1})$ and $(m_{2},n_{2},\ell_{2})$ are shortlists of positive integers such that
\begin{equation}
\label{i=1,2}
F_{m_{1}}+F_{n_{1}}=X_{\ell_{1}} \qquad \text{and} \qquad F_{m_{2}}+F_{n_{2}}=X_{\ell_{2}},
\end{equation}
with $1\le \ell_1< \ell_2$. We also assume that $n_i-m_i \ge 2$ for $i=1,2$.
By the main result in \cite{FT}, we may assume that not both $m_1$ and $m_2$ are zero, although this condition will not be used. Thus, $n_i\ge 2$ for $i=1,2$.

We will start from the assumption that $n \ge 3$. Setting $(m,n,\ell):=(m_{i},n_{i},\ell_{i})$, for $i \in \left\{1,2\right\}$ and
using inequalities \eqref{crecFib} and \eqref{E3}, we get from \eqref{i=1,2} that
\begin{equation*}\label{E4}
\alpha^{n-2}\leq \alpha^{m-2}+\alpha^{n-2} \leq F_{m}+F_{n}=X_{\ell}\leq \delta^{\ell}
\end{equation*}
and
\begin{equation*}
\frac{\delta^{\ell}}{\alpha^2}\leq X_{\ell} = F_{m}+F_{n} \leq \alpha^{m-1} + \alpha^{n-1} \le \alpha^{n-1}\left(1+\alpha^{-2}\right).
\end{equation*}
The above inequalities give
$$
(n-2)\log \alpha<\ell \log \delta\le (n+1)\log\alpha+\log(1+\alpha^{-2}).
$$
Dividing across by $\log \alpha$ and setting $c_{1}:=1/ \log \alpha$, we deduce that
$$
-2<c_1\ell \log \delta - n < 1+\frac{\log(1+\alpha^{-2})}{\log\alpha},
$$
and since $\alpha^2 = \alpha +1 >2$, we get
\begin{equation}
\label{E5}
|n-c_{1}\ell \log \delta| \leq 2.
\end{equation}
Furthermore, $\ell<n$, for if not, we would  then get
$$
\delta^n\le \delta^\ell \le \alpha^{n+1}(1+\alpha^{-2}),\quad {\text{\rm implying}} \quad \left(\frac{\delta}{\alpha}\right)^n \le \alpha+\alpha^{-1},
$$
which is false since $\delta\ge 1+{\sqrt{2}}, \alpha = (1+\sqrt{5})/2$ and $n\ge 3$.

Besides, given that $\ell_1 < \ell_2$, we have by  \eqref{i=1,2} and \eqref{crecFib} that
$$
\A^{n_1-2}\le F_{n_1} \le F_{m_1}+F_{n_1} =X_{\ell_1} < X_{\ell_2} = F_{m_2}+F_{n_2} < \A^{n_2-1}\left(1+\alpha^{-2}\right).
$$
Thus,
\begin{equation}
\label{ineq n1-n2}
n_1 \le n_2+2.
\end{equation}

Using identities \eqref{E2} and \eqref{binet} in the Diophantine equations \eqref{i=1,2}, we get
\[
\cfrac{\alpha^{m}+\alpha^{n}}{\sqrt{5}} - \frac{1}{2}\delta^{\ell} = \frac{1}{2}\eta^{\ell} + \cfrac{\B^{m}+\B^{n}}{\sqrt{5}}
.
\]
Thus, dividing both sides of the above equality by $(\A^n+\A^m)/\sqrt{5}$ and taking absolute value, we get
\begin{equation}
\label{E7}
\left|\delta^{\ell}(\sqrt{5}/2)\alpha^{-n}(1+\alpha^{m-n})^{-1} - 1\right| < \frac{3.6}{\alpha^{n}},
\end{equation}
where we have used the facts that $|\eta| =|\delta|^{-1}, ~ |\B| = \A^{-1}$, $\ell \ge 1$  and $m\ge 0$.

Put
$$
\Lambda_1 := \delta^{\ell}(\sqrt{5}/2)\alpha^{-n}(1+\alpha^{m-n})^{-1}-1,
$$
and
$$
\Gamma_1 := \ell\log \delta+\log(\sqrt{5}/2) -n\log \alpha - \log(1+\alpha^{m-n}).
$$
Since $|e^{\Gamma_1}-1| = |\Lambda_1| < 3.6/\A^n < 0.84$ for $n \ge 3$, it follows that $e^{|\Gamma_1|}<6.25$ and so
\[
|\Gamma_1|<e^{|\Gamma_1|}|e^{\Gamma_1}-1|<\frac{23}{\alpha^{n}}.
\]
Thus, we get
\begin{equation}
\label{E8}
|\ell\log \delta+\log(\sqrt{5}/2) - n\log \alpha - \log(1+\alpha^{m-n})|<\frac{23}{\alpha^{n}}.
\end{equation}

In order to find upper bounds for $n$ and $\ell$, we use a result of
E. M. Matveev on lower bounds for nonzero linear forms in logarithms of algebraic
numbers.

Let $\E$ be an algebraic number of degree $d$ over $\mathbb{Q}$
with minimal primitive polynomial over the integers
$$
f(X) := a_0 \prod_{i=1}^{d}(X-\E^{(i)}) \in \mathbb{Z}[X],
$$
where the leading coefficient $a_0$ is positive. The {\it
logarithmic height} of  $\E$ is given by
$$
h(\E) := \dfrac{1}{d}\left(\log a_0 +
\sum_{i=1}^{d}\log\max\{|\E^{(i)}|,1\}\right).
$$
In this work we will use the following properties. If $\eta=p/q$ is a rational number with $\gcd(p,q)=1$ and $q>0$,
then $h(\eta)=\log \max \{|p|,q\}$. Are also known: $h(\eta^{s}) = |s|h(\eta)$ for all $s\in \Z$ and
\begin{eqnarray*}
h(\eta\pm\gamma) \leq h(\eta) + h(\gamma)+\log 2, \qquad h(\eta\gamma^{\pm 1}) & \leq & h(\eta)+h(\gamma).
\end{eqnarray*}

Our main tool is a lower bound for a linear form in logarithms of algebraic numbers given by the following result of Matveev \cite{matveev}:
\begin{teor}[\bf Matveev's theorem]
\label{TeoMatveev}
Let $\L \subseteq \R$ be a real algebraic number field of degree $d_{\L}$ over
$\Q,\,\,$  $\E_1, \ldots, \E_l$ non--zero elements of $~\L$, and
$d_1, \ldots,  d_l$ rational integers. Put
$$
\Lambda := \E_1^{d_1} \cdots \E_l^{d_l}-1 \qquad \text{and} \qquad
D \geq \max\{|d_1|, \ldots ,|d_l|,3\}.
$$
Let $A_i \geq \max\{d_{\L}h(\E_i), |\log \E_i|, 0.16\}$ be real
numbers, for $i = 1, \ldots, l.$ Then, assuming that $\Lambda \not
= 0$, we have
$$
|\Lambda| > \exp(-3 \times 30^{l+3} \times l^{4.5} \times
d_{\L}^2(1 + \log d_{\L})(1 + \log D)A_1 \cdots A_l).
$$
\end{teor}

\medskip
We apply Matveev's theorem on the left-hand side of \eqref{E8}.
We take $l:=4$,
\begin{eqnarray*}
\eta_{1}&:=&\delta, \qquad\eta_{2}:=\sqrt{5}/2, \qquad\eta_{3}:=\alpha, \qquad\eta_{4}:=1+\alpha^{m-n},\\
    d_{1}&:=&\ell, \qquad d_{2}:=1, \qquad d_{3}:=-n, \qquad d_{4}:=-1.
\end{eqnarray*}
Furthermore, $\mathbb{L}=\mathbb{Q}(\sqrt{d},\sqrt{5})$ which has degree $d_{\mathbb{L}} = 4$. Since $\ell<n$, we take $D:=n$. We have $h(\eta_{1})=(1/2)\log \delta$,
$h(\eta_{2})=(1/2)\log 5$, $h(\eta_{3})=(1/2)\log \alpha$ and
\begin{align*}
h(\eta_{4})&\leq h(1)+h(\alpha^{m-n})+\log 2\\
&=(n-m)h(\alpha)+\log 2\\
&=(n-m)\left(\frac{1}{2}\log \alpha\right)+\log 2.
\end{align*}
Thus, we can take
\[
A_{1}=2\log \delta, \quad A_{2}=2\log 5, \quad A_{3}=2\log\alpha, \quad A_{4}=(2\log\alpha)(n-m)+4\log 2.
\]
Note that $\Gamma_1 \neq 0$, since otherwise
$$
\delta^{\ell} = (2/\sqrt{5})(\alpha^{n}+\alpha^{m}) \in \Q(\sqrt{5}).
$$
But given that $d\neq5$ is squarefree, it follows that $\Q(\sqrt{d})\cap\Q(\sqrt{5}) = \Q$. Hence, $\ell = 0$ and
$\A^n+\A^m = (\sqrt{5}/2)<2$, which is not possible for any $n > m \ge 0$.

Now Matveev's Theorem \ref{TeoMatveev} tells us that
\begin{eqnarray*}
\log |\Gamma_1| & >  & -1.4\cdot 30^{7}4^{4.5}4^{2}(1+\log 4)(1+\log n)(2\log \delta)(2\log 5)(2\log\alpha)\\
              &  \cdot &  ((2\log\alpha)(n-m)+4\log 2)\\
              & > & -7.2 \times 10^{15}(n-m)(\log n)(\log \delta).
\end{eqnarray*}
Comparing the above inequality with \eqref{E7}, we get
\[
n \log \alpha-\log 23 < 7.2 \times 10^{15}(n-m)(\log n)(\log \delta).
\]
Thus,
\begin{equation}
\label{n-(n-m)}
n<1.5 \times 10^{16}(n-m)(\log n)(\log \delta).
\end{equation}
This inequality was under the assumption that $n \ge 3$, but if $n =2$, then the above inequality obviously holds as well

Returning to equation $F_m + F_n = X_{\ell}$, and rewriting it as
$$
\cfrac{\A^{n}}{\sqrt{5}}-\frac{1}{2}\delta^{\ell} = \frac{1}{2}\eta^{\ell}+\cfrac{\B^{n}}{\sqrt{5}}-F_{m},
$$
we obtain
\begin{equation}\label{E71}
\left|\delta^{\ell}(\sqrt{5}/2)\alpha^{-n} - 1\right| < \frac{1.1}{\alpha^{n-m}}.
\end{equation}
Put
\begin{equation*}
\Lambda_2 := \delta^{\ell}(\sqrt{5}/2)\alpha^{-n}-1, \qquad\Gamma_2 := \ell\log \delta+\log(\sqrt{5}/2)-n\log \alpha.
\end{equation*}
Given that $n-m\ge 2$, we have that $|\Lambda_2| = |e^{\Gamma_2}-1| < 0.421$. It follows that
\begin{equation}
\label{FLL-2}
|\ell\log \delta+\log(\sqrt{5}/2)-n\log \alpha| = |\Lambda_2|<e^{|\Gamma_2|}|e^{\Gamma_2}-1|<\frac{2}{\alpha^{n-m}}.
\end{equation}
Furthermore, $\Gamma_2\neq0$, since $\delta^{\ell} \notin \mathbb{Q}(\sqrt{5})$ by  a previous argument and $\A^n>2$ for all $n\ge2$.

Applying Matveev's Theorem \ref{TeoMatveev} with the parameters $l:=3$,
$$
\eta_{1}:=\delta, \quad \eta_{2}:=\sqrt{5}/2, \quad \eta_{3}:=\A, \quad d_{1}:=\ell, \quad d_{2}:=1, \quad d_{3}:=-n,
$$
 we can conclude that
$$
\log|\Gamma_2|>-1.5\cdot10^{14}(\log \delta)(\log n)(\log\alpha),
$$
and comparing with \eqref{FLL-2}, we get
\begin{equation}
\label{eq:nminusm}
n-m < 3.2\cdot10^{14}(\log \delta)(\log n).
\end{equation}
We replace the previous bound \eqref{eq:nminusm} on $n-m$ in \eqref{n-(n-m)} and use the fact that $\delta^{\ell} \le \alpha^{n+1} \left(1+\alpha^{-2}\right)$, to obtain bounds on $n$ and $\ell$ in terms of $\log n$ and $\log\d$.

Let us record what we have proved so far.

\begin{lemma}
Let $(m, n, \ell)$ be a solution of $F_{m} + F_{n} = X_{\ell}$ with $n - m \ge 2$, $m \ge 0$ and $d \neq 5$, then
\begin{equation}\label{l-n-d}
\ell < 1.6 \times 10^{30}(\log n)^2(\log \delta) \quad {\text{and}}\quad n<4.8 \times 10^{30}(\log n)^2(\log \delta)^2.
\end{equation}
\end{lemma}

\section{Absolute bounds}

In this section we will find absolute bounds for $m, n$ and $\ell$, which determine that \eqref{i=1,2} only has a finite number of solutions.

We recall that $(m,n,\ell) = (m_i, n_i, \ell_i)$, where $n_i-m_i\ge2, ~~ m_i\ge0$,  so $n_i\ge 2$, for $i=1, 2$. Moreover, $1 \le \ell_1 < \ell_2$. We return to inequality \eqref{FLL-2} and write:
\begin{equation*}
|\Gamma_2^{(i)}| :=|\ell_i\log \delta + \log(\sqrt{5}/2) - n_i\log\alpha| < \frac{2}{\alpha^{n_i - m_i}}, \quad {\text{\rm for}}\quad i= 1, ~2.
\end{equation*}

We make a suitable cross product between $\Gamma_2^{(1)}, ~\Gamma_2^{(2)}$ and $\ell_1, ~\ell_2$ to eliminate the term involving $\log\delta$ in the above linear forms in logarithms:
\begin{align}
\label{FLL-3}
|\Gamma_3|:=\left|(\ell_2-\ell_1)\log(\sqrt{5}/2) + (\ell_1n_2-\ell_2n_1)\log\alpha \right| &   =  |\ell_2\Gamma_2^{(1)}-\ell_1\Gamma_2^{(2)}|\nonumber\\
                                                                       &   \le \ell_2|\Gamma_2^{(2)}|+\ell_1|\Gamma_2^{(2)}|\nonumber\\
                                                                       &   \le  \frac{2\ell_2}{\alpha^{n_1 - m_1}} + \frac{2\ell_1}{\alpha^{n_2 - m_2}}\nonumber\\
                                                                       &   \le  \frac{4n_2}{\alpha^\lambda}
\end{align}
with $\lambda:=\displaystyle\min_{i=1, 2}\{n_i-m_i\}.$

Next, we apply Matveev's theorem with ~$l=2$,
$$
\eta_{1}:=2a, \quad \eta_{2}:=\alpha, \quad d_{1}:=\ell_1-\ell_2, \quad d_{2}:=\ell_1n_2-\ell_2n_1.
$$
We take $\mathbb{L} := \mathbb{Q}(\sqrt{5})$ and $d_{\mathbb{L}} := 2$. We continue by remarking that $\Gamma_3\neq0$, because $\alpha$ is a unit in the ring of algebraic integers of $\mathbb{Q}(\sqrt{5})$ while the norm of $\sqrt{5}/2$ is $5/4$.

Note that $|\ell_2 - \ell_1| < \ell_2<n_2$. Further, from inequality \eqref{FLL-3}, we have
$$
|\ell_1n_2 - \ell_2n_1| < (\ell_2 - \ell_1)\cfrac{\log(\sqrt{5}/2)}{\log\alpha}+\cfrac{4\ell_2}{\alpha^{\lambda} \log \alpha} < 3.4\ell_2 < 3.4n_2
$$
given that  $\lambda \ge 2$. So, we can take $D := 3.4n_2$.

From Matveev's theorem
$$
\log |\Gamma_3| > -2.6\cdot10^{10}(\log n_2)(\log\alpha).
$$
Combining this with \eqref{FLL-3}, we get
\begin{equation}
\label{Lam-n2}
\lambda < 2.7\cdot10^{10}\log n_2.
\end{equation}
Without loss generality, we can assume that $\lambda=n_i-m_i$, for $i\in\{1,2\}$  fixed.

 We set $\{i,j\}=\{1,2\}$ and return to \eqref{E8} to replace $(m, n, \ell) = (m_i, n_i, \ell_i)$:
\begin{align}
\label{gamma-1}
|\Gamma_1^{(i)}| = | \ell_i\log \delta + \log(\sqrt{5}/2)-n_i\log \alpha - \log(1+\alpha^{-(n_i-m_i)})|<\frac{23}{\alpha^{n_i}}
\end{align}
then to \eqref{FLL-2}, with $(m, n, \ell)=(m_j, n_j, \ell_j)$:
\begin{align}
\label{gamma-2}
|\Gamma_2^{(j)}| =|\ell_j\log \delta + \log(\sqrt{5}/2) - n_j\log \alpha|<\frac{2}{\alpha^{n_j-m_j}}.
\end{align}
We perform a cross product in inequalities \eqref{gamma-1} and \eqref{gamma-2} in order to eliminate the term $\log \delta$:
\begin{eqnarray}
\label{FLL-4}
|\Gamma_4|&:= & \left|(\ell_i-\ell_j)\log(\sqrt{5}/2) + (n_i\ell_j-n_j\ell_i)\log \alpha+\ell_j\log(1+\alpha^{-(n_i-m_i)})\right| \nonumber\\
&   = &  |\ell_i\Gamma_2^{(j)}-\ell_j\Gamma_1^{(i)}| \le \ell_i|\Gamma_2^{(j)}|+\ell_j|\Gamma_1^{(i)}|\le  \frac{25n_2}{\alpha^\rho}
\end{eqnarray}
with $\rho := {\min\{n_i,n_j-m_j\}}$.

If $\Gamma_4 = 0$, we then obtain
$$
 (\sqrt{5}/2)^{\ell_i-\ell_j}= \alpha^{n_i\ell_j-n_j\ell_i}(1+\alpha^{-\lambda})^{\ell_j}.
$$
Since $\alpha$ is a unit, the right--hand side above is an algebraic integer. This is impossible because $\ell_1<\ell_2$ so $\ell_i-\ell_j\ne 0$, and neither $\sqrt{5}/2$ nor $(\sqrt{5}/2)^{-1}$ are algebraic integers. Hence, $\Gamma_4 \neq 0$.

By using Matveev's theorem, with the parameters $l:=3$ and
\begin{eqnarray*}
\eta_{1} &:=& \sqrt{5}/2, \qquad\eta_{2}:=\alpha, \qquad\eta_{3}:=1+\alpha^{-\lambda},\\
  d_{1} &:=&\ell_i-\ell_j,\quad d_{2}:=n_i\ell_j-n_j\ell_i, \quad d_{3}:=\ell_j,
\end{eqnarray*}
and inequalities \eqref{Lam-n2} and \eqref{FLL-4}, we get
$$
\r = \min \{n_i, n_j-m_j\} < 6.8\cdot 10^{12} \lambda \log n_2 < 2\cdot 10^{22}(\log n_2)^2.
$$
Note that  the instance $(i,j) = (2,1)$ leads to $n_1-m_1 \le n_1 \le n_2+2 $ while $(i,j) = (1,2)$ lead to $\rho = \min\{n_1, n_2-m_2\}$.
Hence, either the minimum is $n_1$, so
\begin{equation}
\label{eq:6}
n_1 < 2\cdot 10^{22}(\log n_2)^2,
\end{equation}
or the minimum is $n_j-m_j$ and from inequality \eqref{Lam-n2} we get
\begin{equation}
\label{eq:7}
\max_{i=1,2}\{n_i-m_i\} < 2\cdot 10^{22}(\log n_2)^2.
\end{equation}

Next, assume that we are in case \eqref{eq:7}. We evaluate \eqref{gamma-1} in $i = 1,2$ and make a new cross product in order to eliminate the term involving $\log \delta$:
\begin{eqnarray}
\label{FLL-5}
|\Gamma_5|&:= &  |(\ell_1-\ell_2)\log(\sqrt{5}/2) + (n_1\ell_2 -n_2\ell_1 )\log \alpha\nonumber\\
          &+ & \ell_2 \log(1+\alpha^{m_1-n_1})-\ell_1\log(1+\alpha^{m_2-n_2})|\nonumber\\
&  =  &  |\ell_1\Gamma_1^{(2)} - \ell_2\Gamma_1^{(1)}| \le \ell_1|\Gamma_1^{(2)}|+\ell_2|\Gamma_1^{(1)}|\nonumber\\
          & < & \frac{46 n_2}{\alpha^{n_1-2}}.
\end{eqnarray}
In the above inequality we used inequality \eqref{ineq n1-n2} to conclude that $\min\{n_1, n_2\} \ge n_1-2$.
In order to apply Matveev's theorem we will prove that $\Gamma_5\neq 0$.

\begin{lemma}
The equation
\begin{equation}
\label{eq:1}
(\sqrt{5}/2)^{\ell_2-\ell_1}= \alpha^{m_1\ell_2-m_2\ell_1}(1+\alpha^{n_1-m_1})^{\ell_2}(1+\alpha^{n_2-m_2})^{-\ell_1}
\end{equation}
has no solution in integers $1\le \ell_1<\ell_2$ and $n_i-m_i\ge 2, ~ m_i \ge 0$ for $i=1,2$.
\end{lemma}
\begin{proof}
We let ${\mathbb K}={\mathbb Q}({\sqrt{5}})$. For any positive integer $k$
$$
N_{{\mathbb K}/{\mathbb Q}}(1+\alpha^k)=(1+\alpha^k)(1+\beta^k)=1+(-1)^k+\alpha^k+\beta^k
$$
so
$$
N_{{\mathbb K}/{\mathbb Q}}(1+\alpha^k) = \left\{\begin{matrix} L_k, & {\text{\rm if}} & k\equiv 1\pmod 2,\\
5F_{k/2}^2, & {\text{\rm if}} & k\equiv 2\pmod 4,\\
L_{k/2}^2, & {\text{\rm if}} & k\equiv 0\pmod 4.\end{matrix}\right.
$$
Hence, assuming \eqref{eq:1} and taking norms we get
\begin{eqnarray*}
\left(\frac{-5}{4}\right)^{\ell_2-\ell_1} & = & N_{{\mathbb K}/{\mathbb Q}}(\sqrt{5}/2)^{\ell_2-\ell_1}\\
& = & N_{{\mathbb K}/{\mathbb Q}}(\alpha)^{m_1\ell_2-m_2\ell_1} \frac{N_{{\mathbb K}/{\mathbb Q}}(1+\alpha^{n_1-m_1})^{\ell_2}}{N_{{\mathbb K}/{\mathbb Q}}(1+\alpha^{n_2-m_2})^{\ell_1}}\\
& = & (-1)^{m_1\ell_2-m_2\ell_1} \frac{E_{n_1-m_1}^{\ell_2}}{E_{n_2-m_2}^{\ell_1}},
\end{eqnarray*}
where $E_k\in \{L_k, L_{k/2}^2, 5F_{k/2}^2\}$ according to the residue class of $k$ modulo $4$. If $n_1-m_1=n_2-m_2$, then the right-hand side is $E_{n_1-m_1}^{\ell_2-\ell_1}$, which is an integer. This is impossible since
the left--hand side is not an integer. So, $n_1-m_1\ne n_2-m_2$. In the left, we have $5$ in the numerator. Thus, we must have $5$ in the numerator in the right as well. Since $5\nmid L_k$ for any $k$,  it follows that
$n_1-m_1\equiv 2\pmod 4$ and $E_{n_1-m_1}=5F_{(n_1-m_1)/2}^2$. Thus, the exponent of $5$ in $E_{n_1-m_1}^{\ell_2}$ is at least $\ell_2$. Since it is $\ell_2-\ell_1<\ell_2$ in the left it follows that $5\mid E_{n_2-m_2}$.
By the previous argument, $n_2-m_2\equiv 2\pmod 4$ and $E_{n_2-m_2}=5F_{(n_2-m_2)/2}^2$. By the Carmichael primitive divisor theorem, if $\ell\ge 7$ is odd, $F_{\ell}$ has a primitive prime factor $p$ which exceeds $5$ and does not divide $F_m$ for any $m<\ell$. Using this theorem, we conclude easily that $(n_1-m_1)/2\le 5$ and $(n_2-m_2)/2\le 5$ (otherwise, since $(n_1-m_1)/2\ne (n_2-m_2)/2$ are odd, the fraction
$(5F_{(n_1-m_1)/2})^{\ell_2}/(5F_{(n_2-m_2)/2}^2)^{\ell_1}$ in reduced form will contain with positive or negative exponent a primitive prime $p>5$ of $F_k$, where $k=\max\{(n_1-m_1)/2,(n_2-m_2)/2\}$, which
does not appear in the left).

Assume that one of $(n_1-m_1)/2$ or $(n_1-m_2)/2$ is $5$. Then the exponent of $5$ is one of $3\ell_2$ (if $(n_1-m_1)/2=5$), or $\ell_2-3\ell_1$ (if $(n_2-m_2)/2=5$) and none of these equals $\ell_2-\ell_1$ which is the exponent of $5$ on the left. Hence, $\{(n_1-m_1)/2,(n_2-m_2)/2\}= \{1,3\}$. Since the exponent of $2$ appears with negative sign in the left, we conclude that the only possibility is
$$
(n_1-m_1)/2=1,\qquad (n_2-m_2)/2=3.
$$
In this case, we get
$$
\left(\frac{-5}{4}\right)^{\ell_2-\ell_1}=(-1)^{m_1\ell_2-m_2\ell_1} \left(\frac{5^{\ell_2}}{(5\cdot 4)^{\ell_1}}\right)=\pm \frac{ 5^{\ell_2-\ell_1}}{4^{\ell_1}},
$$
and comparing the exponents of $2$ in both sides we get $\ell_2-\ell_1=\ell_1$, so $\ell_2=2\ell_1$. We now return to equation \eqref{eq:1} and use
$$
1+\alpha^2={\sqrt{5}} \alpha\qquad {\text{\rm and}}\qquad 1+\alpha^6=  2{\sqrt{5}}\alpha^3,
$$
to get
\begin{eqnarray*}
\left(\frac{{\sqrt{5}}}{2}\right)^{\ell_1} & = & \left(\frac{{\sqrt{5}}}{2}\right)^{\ell_2-\ell_1}=\alpha^{m_1\ell_2-m_2\ell_1} \left(\frac{({\sqrt{5}} \alpha)^{\ell_2}}{(2{\sqrt{5}}\alpha^3)^{\ell_1}}\right)\\
& = & \alpha^{\ell_1(2m_1-m_2)} {\sqrt{5}}^{\ell_2-\ell_1} \alpha^{\ell_2-3\ell_1} 2^{-\ell_1}=\left(\alpha^{2m_1-m_2-1} \left(\frac{{\sqrt{5}}}{2}\right)\right)^{\ell_1}.
\end{eqnarray*}
Extracting $\ell_1$ powers, we get that
$$
\frac{{\sqrt{5}}}{2}=\zeta \alpha^{2m_1-m_2-1} \frac{{\sqrt{5}}}{2},
$$
where $\zeta$ is some root of unity of order $2\ell_1$.  Hence, $\alpha^{2m_1-m_2-1}=\zeta^{-1}$. Since the left--hand side is real and positive, we have $\zeta^{-1}=1$ so $2m_1-m_2-1=0$. Since $n_1=m_1+2,~n_2=m_2+6$,
we get that
$$
n_2=m_2+6=(2m_1-1)+6=2(n_1-2)+5=2n_1+1.
$$
Hence, putting $k=n_1$, we have gotten to the situation where
\begin{eqnarray*}
X_{\ell_1} & = & F_{n_1}+F_{m_1}=F_{n_1}+F_{n_1-2},\\
X_{\ell_2} & = & X_{2\ell_1}=F_{n_2}+F_{m_2}=F_{2n_1+1}+F_{2n_1-5}.
\end{eqnarray*}
Since
$$
X_{2\ell_1}=2X_{\ell_1}^2\pm 1,
$$
we get
$$
F_{2n_1+1}+F_{2n_1-5}=X_{2\ell_1}=2X_{\ell_1}^2\pm 1=2(F_{n_1}+F_{n_1-2})^2\pm 1.
$$
Of course this is absurd because the left--hand side is always even and the right--hand side is always odd. Thus, it is not possible that
$$
F_{2n_1+1}+F_{2n_1-5}-2(F_{n_1}+F_{n_1-2})^2=\pm 1.
$$
It is an easy exercise though to show that the left--hand side above is $\pm 4$ (namely, $(-1)^{n_1} 4$ for all $n_1\ge 3$).
\end{proof}
We apply a linear form in four logarithms to obtain an upper bound to $n_1$. We take
\begin{eqnarray*}
\eta_{1}&:=&\sqrt{5}/2, \qquad\eta_{2}:=\alpha,\qquad\eta_{3}:=1+\alpha^{m_1-n_1}, \qquad \eta_{4}:=1+\alpha^{m_2-n_2},\\
d_{1}&:=&\ell_1-\ell_2, \quad d_{2}:=n_1\ell_2-n_2\ell_1, \quad d_{3}:=\ell_2, \quad d_{4}:=-\ell_1,
\end{eqnarray*}
and apply Matveev's theorem on the left--hand side of inequalities \eqref{FLL-5}. Combining the resulting inequality with
the right--hand side in \eqref{FLL-5} and inequalities \eqref{Lam-n2} and \eqref{eq:7}
 leads us to
\begin{eqnarray}
\label{eq:8}
n_1 & < & 2.12\cdot10^{15} h(1+\alpha^{m_1-n_1}) h(1+\alpha^{m_2-n_2})(\log n_2)\nonumber\\
       & ~ &~~~~~~ < 5\cdot10^{14}(n_1-m_1)(n_2-m_2)(\log n_2)\nonumber\\
			 & ~ &~~~~~~~~~~~~~ < 2.7\cdot10^{47}(\log n_2)^4.
\end{eqnarray}
Thus, we have that inequality \eqref{eq:8} holds provided that  \eqref{eq:7} holds. Otherwise, \eqref{eq:6} holds which is even better than \eqref{eq:8}.
Hence, we conclude that $n_1 < 2.7\cdot10^{47}(\log n_2)^4$ holds in all cases.

By inequality \eqref{E5},
$$
\log\delta \le \ell_1\log \delta \le (n_1+1)\log\alpha + \log(1+\alpha^{-2}) < 1.3\cdot10^{47}(\log n_2)^4.
$$
Putting this into \eqref{l-n-d} we get $n_2 < 8.2\cdot10^{124}(\log n_2)^{10}$, and then $n_2 < 2.1\cdot10^{150}$.

In summary, we have proved the following result.

\begin{lemma}
\label{cota:m-n}
Let $(m_i, n_i, \ell_i)$ be two solutions of $F_{m_i} + F_{n_i} = X_{\ell_i}$ for $i=1,2$, with $n_i - m_i \ge 2$, $m_i \ge 0$, $d \neq 5$, $1\le \ell_1<\ell_2$, then
\begin{equation*}
\max\{m_1, \ell_1\} < n_1 < 4\cdot 10^{57} \qquad {\text{and}}\qquad \max\{m_2, \ell_2\}< n_2 < 2.1\cdot10^{150}.
\end{equation*}
\end{lemma}

\section{Reducing $n_1$ and $n_2$}

\noindent
In the above Lemma \ref{cota:m-n}, we obtained upper bounds on our variables which are very large, so we need to reduce them. With this aim, we use some results from the theory of continued fractions and the geometry of numbers.

The following results, well--known in the theory of Diophantine approximation, will be used for the treatment of linear forms homogeneous in two integer variables.

\begin{lemma}
\label{met. red. 1}
Let  $\tau$ be an irrational number, $M$ be a positive integer and $p_0/q_0,p_1/q_1,,\ldots$ be all  the convergents of the continued fraction of $\tau$. Let $N$ be such that $q_N>M$.
Then putting
$$
a(M) := \max\{a_t : t = 0, 1,\ldots, N\},\quad {\text{the~inequality}} \quad  |m\tau-n|> \cfrac{1}{(a(M) + 2)m},
$$
holds for all pairs $(n,m)$ of integers with $0< m < M$.
\end{lemma}
For the treatment of nonhomogeneous linear forms in two integer variables, we will use a slight variation of a result due to Dujella and Peth\H{o},
which itself is a generalization of a result of Baker and Davenport (see \cite{Dujella-Petho}). For a real number $X$, we put
$$
||X|| := \min\{|X-n|\,:\, n \in \Z\}
$$
for the distance from $X$ to the nearest integer.
\begin{lemma}
\label{Dujella-Petho}
Let $\tau$ be an irrational number, $M$ be a positive integer, and $p/q$ be a convergent of the continued
fraction of the irrational $\tau$ such that $q > 6M$. Let $A, B,  \mu$ be some real
numbers with $A > 0$ and $B > 1$. Put $\epsilon :=||\mu q|| - M||\tau q||$. If $\epsilon > 0$, then there is no solution to the
inequality
$$
0 < | m\tau - n + \mu | < AB^{-k},
$$
in positive integers $m, n$ and $k$ with
$$
m \leq M \qquad {\text{and}} \qquad k \geq \dfrac{\log(Aq/\epsilon)}{\log B}.
$$
\end{lemma}
At various occasions, we need to find a lower bound for linear forms with bounded integer coefficients (in three and four integer variables).
Let $\tau_1, \ldots,\tau_t \in \R$ and the linear form
\begin{eqnarray}
\label{lll}
x_1\tau_1+ x_2\tau_2+\cdots+x_t\tau_t \quad  {\text{\rm with}} \quad |x_i|\le X_i.
\end{eqnarray}
We set $X:=\max\{X_i\}, ~C>(tX)^t$ and consider the integer lattice $\Omega$ generated by
\begin{eqnarray*}
\label{def-bi}
{\bf b}_j := {\bf e}_j+\left\lfloor C\tau_j
\right\rceil{\bf e}_t \qquad {\rm for}\qquad 1\le j \le t-1 \qquad {\rm and}
\qquad {\bf b}_t := \left\lfloor C\tau_t\right\rceil{\bf e}_t,
\end{eqnarray*}
where $C$ is a sufficiently large positive constant.
\begin{lemma}
\label{flacotadas} Let $X_1, \ldots, X_t$ be positive integers
such that $X:=\max\{X_i\}$ and $C>(tX)^t$ is a fixed constant. With the above notation on $\Omega$, we consider a reduced base
$\{{\bf b}_i\}$ to $\Omega$ and its base of Gram--Schmidt $\{{\bf
b}_i^{*}\}$ associated. We set
$$
c_1 := \displaystyle\max_{1\le i\le t} \dfrac{||
\mathbf{b}_1||}{||\mathbf{b}_i^{*}||}, \quad \delta :=
\dfrac{||\mathbf{b}_1||}{c_1}, \quad Q := \sum_{i=1}^{t-1}
X_i^2\qquad {and}\qquad T := \left(1+\sum_{i=1}^{t} X_i\right)/2.
$$
If the integers $x_i$ satisfy that $|x_i| \le X_i$, for
$i=1,\ldots,t$ and $\delta^2 \ge T^2 + Q$, then we have
\begin{eqnarray*}
\left|\sum_{i=1}^{t} x_i\tau_i\right| \ge \cfrac{\sqrt{\delta^2 - Q} - T}{C}.
\end{eqnarray*}
\end{lemma}
\noindent For more details, see Proposition 2.3.20 in \cite[Section 2.3.5]{Cohen}.

\subsection{First reduction}

With this purpose of reducing the upper bound to $n_1$ and $n_2$ given in Lemma \ref{cota:m-n} to cases
that can be treated computationally, we return to $\G_3, \G_4$ and $\G_5$.

Dividing both sides of inequality \eqref{FLL-3} by $(\ell_2-\ell_1)\log \A$, we obtain
\begin{equation}\label{red-1}
\begin{split}
\left|\dfrac{\log(\sqrt{5}/2)}{\log \A} - \dfrac{\ell_2n_1-\ell_1n_2}{\ell_2-\ell_1}\right| & < \frac{8.4n_2}{\alpha^\lambda(\ell_2-\ell_1)} \quad {\text{\rm with }}\quad \lambda:=\displaystyle\min_{i=1, 2}\{n_i-m_i\}.
\end{split}
\end{equation}
Bellow we apply Lemma \ref{met. red. 1}. We put $\tau:=\log(\sqrt{5}/2)/\log \A$ (which is an irrational) and compute its continued fraction $[a_0,a_1,a_2,\ldots]$ and its convergents $p_1/q_1 , p_2/q_2, \ldots $
$$
[0, 4, 3, 5, 7, 3, 1, 8, 45, 1, 3, 1,\ldots]\qquad{\rm and}\qquad 0, ~\frac{1}{4}, ~\frac{3}{13}, ~\frac{16}{69}, ~\frac{115}{496}, ~\frac{361}{1557}, ~\frac{476}{2053}, \ldots.
$$
Furthermore, we note that taking $M:=2.1\cdot10^{150}$ (according to Lemma \ref{cota:m-n}), it follows that
$$
q_{282} > M > n_2 > \ell_2-\ell_1 \quad {\rm and} \quad a(M) := \max\{a_i \, : \, 0 \le i \le 282\} = 258.
$$
Then, by Lemma \ref{met. red. 1}, we have that
\begin{equation}
\label{red-1:2}
\begin{split}
\left|\tau - \dfrac{\ell_2n_1-\ell_1n_2}{\ell_2-\ell_1}\right|> \dfrac{1}{260(\ell_2-\ell_1)^2}.
\end{split}
\end{equation}
Hence, combining the inequalities \eqref{red-1} and \eqref{red-1:2}, we obtain
$$
\A^{\l}<2184\cdot n_2(\ell_2-\ell_1) < 9.7\cdot10^{303},
$$
so $\l \le 1455$.

Now, for each $n_i-m_i = \l \in[2, ~ 1455]$ we estimate (via LLL--algorithm) a lower bound for $|\Gamma_4|$, with
\begin{equation}
\label{red-2}
\Gamma_4 = (\ell_i-\ell_j)\log(\sqrt{5}/2) + (n_i\ell_j-n_j\ell_i)\log \alpha+\ell_j\log(1+\alpha^{m_i-n_i})
\end{equation}
given in inequality \eqref{FLL-4}. Recall that $\Gamma_4 \neq 0$.

We put as in \eqref{lll}, $t:=3$,
$$
\tau_1:=\log(\sqrt{5}/2),\qquad\qquad\tau_2:=\log\A,\qquad\qquad\tau_3:=\log(1+\A^{-\l}),
$$
and
$$
x_1:=\ell_i - \ell_j,\qquad\qquad x_2:=n_i\ell_j-n_j\ell_i,\qquad\qquad x_3:=\ell_j.
$$
Further, we set $X := 7.2\cdot10^{150}$ as an upper bound to $|x_i|<3.4n_2$ for all  $i = 1, 2, 3$, and $C:=(20X)^5$.
A computer search allows us to conclude, together with inequality \eqref{FLL-4}, that
\begin{equation*}
10^{-608}<\min_{\substack{\lambda\in [2,~1455]}}|\Gamma_4| < 25n_2\cdot\A^{-\rho},\qquad{\rm with}~~~~\rho := {\min\{n_i,n_j-m_j\}},
\end{equation*}
which leads to $\rho \le 3635$. As we noted before, $\rho = n_1$ (so $n_1 \le 3635$), or $\rho = n_j-m_j$.

Next we suppose that $n_j-m_j = \rho \le 3635$. Since $\lambda \le 1455$, we have
$$
\lambda =\min_{i=1,2}\{n_i-m_i\} \le 1455 \qquad {\text{\rm and}}\qquad \chi:= \max_{i=1,2}\{n_i-m_i\} \le 3635.
$$
Returning to inequality \eqref{FLL-5} which involves
\begin{eqnarray}
\label{red-3-0}
\Gamma_5&:= &  (\ell_1-\ell_2)\log(\sqrt{5}/2) + (n_1\ell_2 -n_2\ell_1 )\log \alpha\nonumber\\
          & + &\ell_2 \log(1+\alpha^{m_1-n_1})-\ell_1\log(1+\alpha^{m_2-n_2})\neq 0,
\end{eqnarray}
we use again the LLL--algorithm to estimate a lower bound for $|\Gamma_5|$ and so to find a beter bound to $n_1$
 than the one given in Lemma \ref{cota:m-n}.

We will distinguish the cases $\lambda < \chi$ or $\lambda = \chi$.

\medskip

\noindent {\bf The case $\lambda < \chi$}.

\medskip

We take $\lambda \in [2,1455]$ and $ \chi \in[\lambda + 1, 3635]$ and put for \eqref{lll}, $t:=4$,
$$
\tau_1:=\log(\sqrt{5}/2),\qquad\tau_2:=\log\A,\qquad\tau_3:=\log(1+\A^{m_1-n_1}),\qquad\tau_4:=\log(1+\A^{m_2-n_2}),
$$
and
$$
x_1:=\ell_1 - \ell_2,\qquad x_2:=n_1\ell_2-n_2\ell_1,\qquad x_3:=\ell_2, ,\qquad x_4:=-\ell_1.
$$
Also we put $X := 7.2\cdot10^{150}$ and $C:=(7X)^9$. Computationally we confirm that,
\begin{equation*}
10^{-1215}<\min_{\substack{\lambda\in [2,1455]\\ \chi\in[\lambda+1,3635]}}|\Gamma_5|,
\end{equation*}
which together with inequality \eqref{FLL-5} lead to inequality
$$
\A^{n_1-2}<46\cdot10^{1215}n_2.
$$
Hence, considering the upper bound on $n_2$ given in Lemma \ref{cota:m-n}, we conclude that $n_1 \le 6545.$

\medskip

\noindent {\bf The case $\lambda = \chi$}.

\medskip

In this case, we have
\begin{align*}
\Gamma_5&:= (\ell_2-\ell_1)\left(\log(2/\sqrt{5})+\log(1+\alpha^{m_1-n_1})\right) + (n_1\ell_2 -n_2\ell_1 )\log \alpha.
\end{align*}
We divide inequality \eqref{FLL-5} by $(\ell_2-\ell_1)\log\A$ to obtain
\begin{equation}
\label{red-3}
\begin{split}
\left|\dfrac{\left|\log(2/\sqrt{5})+\log(1+\alpha^{m_1-n_1})\right|}{\log \A} - \dfrac{\ell_2n_1-\ell_1n_2}{\ell_2-\ell_1}\right| & < \frac{96n_2}{\alpha^{n_1-2}(\ell_2-\ell_1)}.
\end{split}
\end{equation}
We now put $ \tau_{\lambda}:=\left|\log(2/\sqrt{5})+\log(1+\alpha^{\lambda})\right|/\log \A$ and compute its continued fractions $[a_0^{(\lambda)},a_1^{(\lambda)},a_2^{(\lambda)},\ldots]$ and its convergents $p_1^{(\lambda)}/q_1^{(\lambda)},p_2^{(\lambda)}/q_2^{(\lambda)},\ldots$ for each $\lambda \in [2, 1455]$.
Furthermore, for each case we find an integer $t_{\lambda}$ such that $q_{t_{\lambda}}^{(\lambda)} > 2.1\cdot10^{150} > n_2 > \ell_2-\ell_1 $ and calculate
$$
a(M) := \max_{2\le \lambda \le 1455}\{a_i^{(\lambda)} \, : \, 0\le i \le t_{\lambda}\}.
$$
A simple computational routine in Mathematica reveals that for $\lambda = 312$, $t_{\lambda} = 270$ and $i= 223$ we have $a(M)= a_{223}^{(312)} = 1000002$.
Hence, combining the concusion of Lemma \ref{met. red. 1} and inequality \eqref{red-3}, we get $\alpha^{n_1-2} < 96\cdot1000004 n_2 (\ell_2-\ell_1) < 4.3\cdot10^{308},$ so
$n_1 \le 1170$.

Hence, we obtain that $n_1 \le 6545$ holds in all cases ($\rho = n_1$, $\lambda < \chi$ or $\lambda = \chi$).

\noindent By inequality \eqref{E5},
$$
\log\delta \le \ell_1\log \delta \le n_1\log\alpha + \log(1+\alpha^{-2}) < 3150.
$$
Considering the above inequality in \eqref{l-n-d} we conclude that $n_2 < 3\cdot10^{37}(\log n_2)^{2}$ which yield $n_2 < 4.4\cdot10^{41}$.
In summary, after this first cycle of reduction, we have
\begin{equation}
\label{new-n2}
\begin{split}
n_1 \le 6545 \qquad {\text{\rm and}}\qquad  n_2 < 4.4\cdot10^{41}.
\end{split}
\end{equation}
We note that the above upper bound for $n_2$ represents a very good reduction of the bound given in Lemma \ref{cota:m-n}. Hence,
it is expected that if we restart our reduction cycle with our new bound on $n_2$, then we can get an even better bound on $n_1$.
Indeed, returning to \eqref{red-1}, we take $M:=4.4\cdot10^{41}$ and computationally we verify that $q_{89} > M > n_2 > \ell_2-\ell_1$ and
$a(M) := \max\{a_i \, : \, 0 \le i \le 89\} = a_{73} = 161$, from which it follows that $\lambda \le 414$. We now return to \eqref{red-2},
where putting $X:=1.5\cdot10^{42}$ and $C:=(7X)^5$, we apply LLL-algorithm to $\lambda \in[2, 414]$. This time we get
$7.9\cdot10^{-174}<\min_{\substack{\lambda\in [2, 414]}}|\Gamma_4|$, then $\rho \le 1035$. Continuing under the assumption $n_j-m_j = \rho \le 1035$,
we return to \eqref{red-3-0} and put $X:=1.5\cdot10^{42}$, $C:=(11X)^9$ and $M:=4.4\cdot10^{41}$ for the cases $\lambda < \chi$  and $\lambda = \chi$.
One can confirm computationally that
\begin{equation*}
2.7\cdot10^{-347}<\min_{\substack{\lambda\in [2,414]\\ \chi\in[\lambda+1,1035]}}|\Gamma_5| \qquad {\rm  and}\qquad a(M)= a_{45}^{(43)} =19362,
\end{equation*}
respectively and thus we obtain $n_1 \le 1870$. Running one more time the reduction cycle, we concluded that $n_1 \le 1811$.

In the next lemma we summarize the reductions achieved.
\begin{lemma}
\label{cota:m1-l1-n2}
Let $(m_i, n_i, \ell_i)$ be two solutions of $F_{m_i} + F_{n_i} = X_{\ell_i}$ for $i=1,2$, with $n_i - m_i \ge 2$, $m_i \ge 0$, $d \neq 5$, $1\le \ell_1<\ell_2$, then
$$
m_1 < n_1 \le 1811, \quad \ell_1 \le 990  \quad {\it and} \quad n_2 < 3.3\cdot10^{40}.
$$
\end{lemma}

\subsection{Final reduction.}

From \eqref{delta-eta} and \eqref{E2} and the fact that  $(X_{1},Y_{1})$ is the smallest positive integer solution to the Pell equation $X^{2}- d Y^{2}=\pm 1$, we obtain
\begin{eqnarray*}
X_{\ell} & = & \frac{1}{2}\left(\delta^{\ell}+\eta^{\ell}\right) =  \frac{1}{2}\left(\left( X_{1}+\sqrt{d}Y_{1}\right)^{\ell} + \left(X_{1}-\sqrt{d}Y_{1}\right)^{\ell}\right)\\
         & = & \frac{1}{2}\left(\left(X_{1}+\sqrt{X_1^2 \mp 1}\right)^{\ell}+\left(X_{1}-\sqrt{X_1^2 \mp 1}\right)^{\ell}\right):= P_{\ell}^{\pm}(X_1).
\end{eqnarray*}

Thus, returning to the equation $F_{m_{1}}+F_{n_{1}} = X_{\ell_1}$, we consider the equations:
\begin{eqnarray}
\label{eqf}
P_{\ell_1}^{+}(X_1) = F_{m_{1}} + F_{n_{1}} \qquad {\rm and } \qquad P_{\ell_1}^{-}(X_1) = F_{m_{1}} + F_{n_{1}},
\end{eqnarray}
with $m_1 \in [0, 1811],~n_1 \in [m_1+2, 1811]$ and $\ell_1 \in [1, 990]$.

 A computer search on the above equations \eqref{eqf}  shows that
\scriptsize
\begin{center}
$P_{\ell_1}^{+}$: \qquad
\begin{tabular}{ c  c  c  c  c | c  c  c  c  c }

 $(n_1, m_1, \ell_1)$ & $X_1$ & $d$ & $Y_1$ & $\delta$ & $(n_1, m_1, \ell_1)$ & $X_1$ & $d$ & $Y_1$ & $\delta$ \\
\hline
                  (5, 3, 2) & 2 & 3 & 1 & $ 2 + \sqrt{3}$ & (12, 10, 2) & 10 & 11 & 3 & $ 10 + 3\sqrt{11}$\\

                  (8, 5, 3) & 2 & 3 & 1 & $ 2 + \sqrt{3}$ & (13, 6, 2) & 11 & 30 & 2 & $ 11 +2\sqrt{30}$\\

                  (11, 6, 2) & 7 & 12 & 2 & $ 7 + 4\sqrt{3}$ & (21, 5, 2) & 74 & 219 & 5 & $ 74 + 5\sqrt{219}$\\

                  (11, 6, 4) & 2 & 3 & 1 & $ 2 + \sqrt{3}$ &\\

\hline
\end{tabular}
\end{center}
\normalsize
It easy to see that $(n_1, m_1, \ell_1, X_1) = (2, 0, \ell_1, 1)$ too are solutions for all $\ell_1\in[1, 990]$. However, these lead to $Y_1 = 0$, which is not of interest to us. On the other hand
\scriptsize
\begin{center}
$P_{\ell_1}^{-}$: \qquad
\begin{tabular}{ c  c  c  c  c  c }

                  $(n_1, m_1, \ell_1)$ & $X_1$ & $d$ & $Y_1$ & $\delta$ \\
\hline
                  (4, 0, 2) & 1 & 2 & 1 & $ 1 +\sqrt{2}$ \\

                  (3, 1, 2) & 1 & 2 & 1 & $ 1 +\sqrt{2}$ \\

                  (5, 3, 3) & 1 & 2 & 1 & $ 1 + \sqrt{2}$\\

                  (23, 12, 2) & 120 & 14401 & 1 & $ 120 +\sqrt{14401}$\\
\hline
\end{tabular}
\end{center}
\normalsize
are the only solutions. We note that $7+4\sqrt{3}=(1+\sqrt{3})^2$, so these come from the same Pell equation with $d = 3$.

From the above tables, we are let to set
\begin{eqnarray*}
&&
\delta_1: = 2+\sqrt{3}, \qquad\qquad \delta_2: = 10 + 3\sqrt{11}, \qquad\qquad \delta_3 := 11 + 2\sqrt{30}, \\
&& \delta_4 := 74 + 5\sqrt{219}, \qquad \delta_5: = 1+\sqrt{2}, ~~~\qquad\quad\qquad \delta_6 := 120+\sqrt{14401}.
\end{eqnarray*}

We work on the linear form in logarithms $\Gamma_1$ and $\Gamma_2$, in order to reduce the upper bound on $n_2$ given in Lemma \ref{cota:m1-l1-n2}.
From inequality \eqref{FLL-2}, for $(m, n, \ell)=(m_2, n_2, \ell_2)$, we write
\begin{equation}
\label{D-P}
\left|\ell_2\dfrac{\log \delta_s}{\log\A}-n_2 + \cfrac{\log(\sqrt{5}/2)}{\log\A}\right|< 4.2\cdot\A^{-(n_2-m_2)},~~~{\rm for}~~s=1, 2, \ldots, 6.
\end{equation}
We put
$$
\tau_s:=\dfrac{\log\delta_s}{\log\A},\quad \quad \mu_s:= \dfrac{\log (\sqrt{5}/2)}{\log\A} \qquad
{\rm and} \qquad A_s:= 4.2, \quad \quad  B_s:= \A.
$$
By the Gelfond-Schneider's theorem, we conclude that $\tau_s$ is transcendental (so irrational).
Inequality \eqref{D-P} can be rewritten as
\begin{equation}
\label{D-P-2}
0<|\ell_2\tau_s-n_2+\mu_s| < A_sB_s^{-(n_2-m_2)},~~~{\rm for}~~s=1, 2, \ldots, 6.
\end{equation}
Now, we take $M:= 3.3\times 10^{40}$ which is an upper bound on $n_2$ (according to Lemma \ref{cota:m1-l1-n2}), and apply Lemma \ref{Dujella-Petho} to inequality \eqref{D-P-2}.
 For each $\tau_s$ with $s =1, \ldots, 6$, we compute its continued fraction $[a_0^{(s)},a_1^{(s)},a_2^{(s)},\ldots]$ and its convergents $p_1^{(s)}/q_1^{(s)},p_2^{(s)}/q_2^{(s)},\ldots$.

In each case, by means of computer search with Mathematica, we find an integer $t_s$ such that
$$
q_{t_s}^{(s)} > 2\times10^{41} =6M  ~\qquad {\rm and} \qquad \epsilon_s :=||\mu_s q^{(s)}|| - M||\tau_s q^{(s)}|| >0.
$$

Finally we found the values of $h_s:=\lfloor\log(A_s q_{t_2}^s/\epsilon_s)/\log B_s\rfloor$:

\begin{center}
\begin{tabular}{ c | c  c  c  c  c  c }
$s$    & 1 & 2 & 3 & 4 & 5 & 6 \\
$t_s$  & 73 & 74 & 98 & 86 & 85 & 81 \\
$\epsilon_s$ & $> 0.34$ & $> 0.24$ & $> 0.35$  & $> 0.38$ & $> 0.09$ & $0.37$  \\
$h_s$ & 204 & 204 & 203 & 206 & 209 & 203 \\
\end{tabular}.
\end{center}
Hence, the above $h_s$ correspond to upper bounds on $n_2-m_2$, for each $s =1, \ldots, 6$, according to Lemma \ref{Dujella-Petho}.

Replacing $(m, n, \ell)=(m_2, n_2, \ell_2)$ in inequality \eqref{E8}, we can write
\begin{equation}
\label{D-P2}
\left|\ell_2\dfrac{\log \delta_s}{\log\A}-n_2 + \cfrac{\log\left((\sqrt{5}/2)/\left(1+\A^{-(n_2-m_2)}\right)\right)}{\log\A}\right|< 47.8\cdot\A^{-n_2},~~~{\rm for}~~s=1, 2, \ldots, 6.
\end{equation}
We now put
$$
\tau_s:=\dfrac{\log\delta_s}{\log\A},\quad \quad  A_s:= 47.8, \quad \quad  B_s:= \A
$$
and
$$
\mu_{s, n_2-m_2}:= \cfrac{\log\left((\sqrt{5}/2)/\left(1+\A^{-(n_2-m_2)}\right)\right)}{\log\A}.
$$
With the above parameters we rewrite \eqref{D-P2} as
\begin{equation}
\label{D-P2-2}
0<|\ell_2\tau_s-n_2+\mu_{s, n_2-m_2}| < A_sB_s^{-n_2},~~~{\rm for}~~s=1, 2, \ldots, 6.
\end{equation}
Bellow we apply again Lemma \ref{Dujella-Petho} to the above inequality \eqref{D-P2-2}, for
$$
s=1, \ldots, 6 \qquad {\rm and} \qquad n_2-m_2 \in [1, d_s], \qquad {\rm with} \qquad M:= 3.3\times 10^{40}.
$$
Taking
$$
\epsilon_{s, n_2-m_2} := ||\mu_s q^{(s, n_2-m_2)}|| - M||\tau_s q^{(s, n_2-m_2)}||,
$$
and
$$
h_{s, n_2-m_2} := \lfloor\log(A_s q^{(s, n_2-m_2)}/\epsilon_{s, n_2-m_2})/\log B_s\rfloor,
$$
we obtain computationally that
$$
\max\{h_{s, n_2-m_2} ~:~ s = 1, \ldots, 6 \quad {\rm and} \quad n_2-m_2 = 1, \ldots, h_s\} \le 227.
$$
Thus, by Lemma \ref{Dujella-Petho}, we have $n_2\le 227$, for all $s = 1,\ldots,6$. Running a new reduction cycle from inequality \eqref{D-P2}, with $M:= 227$ (new upper bound on $n_2$), we finally obtain $n_2 \le 42$ and by inequality \eqref{ineq n1-n2} we have $n_1\le n_2+2$. Given that $\delta^{\ell_2}\le 2\A^{n_2}$ we conclude that $\ell_1 < \ell_2 \le 25$.
Gathering all the information obtained, our problem is reduced to search solutions for \eqref{i=1,2} in the following range:
\begin{equation}
\label{DioTrib-red1}
1 \le \ell_1 < \ell_2 \le 25, \quad m_2 +2 \le n_2 \in [2, 42] \qquad {\rm and} \qquad m_1 +2  \le n_1 \in [2, 44].
\end{equation}
Checking equalities \eqref{i=1,2} in the above range, we obtain the following solutions.

For $\epsilon = +1$:
$$
F_{3}+F_{5} = 7 = X_2,\qquad F_5+F_8= 26 = X_3, \qquad F_6+F_{11} = 97 = X_4 \qquad(\delta = 2+\sqrt{3})
$$
$$
F_3+F_6= 10 = X_1, \qquad F_{10}+F_{12}= 199 = X_2,  \qquad (\delta = 10+3\sqrt{11})
$$
$$
F_4+F_6= 11 = X_1, \qquad F_6+F_{13}= 241 = X_2 \qquad (\delta = 11+2\sqrt{30})
$$
and
$$
F_{5}+F_{21} = 10951 = X_2 \qquad (\delta = 74+5\sqrt{219}).
$$
The above table contains only the information on $X_{\ell}=F_n+F_m$ with $n-m\ge 2$, but we can find the additional solutions
when $n-m\le 1$. Indeed, they are
$$
F_3+F_0=2F_1=2F_2=2=X_1,\qquad 2F_7=26=X_3\qquad (\delta=2+{\sqrt{3}})
$$
$$
2F_5=10=X_1\qquad (\delta=10+3{\sqrt{11}}).
$$

For $\epsilon = -1$:
$$
F_1+F_3 = 3 = X_2, \qquad F_3+F_5= 7 = X_3, \qquad (\delta = 1+\sqrt{2})
$$
$$
F_{12}+F_{23}= 28801 = X_2 \qquad (\delta = 120+\sqrt{14401}).
$$
Allowing for $m=0$ or $m\in \{n-1,n\}$, we get the additional solutions $F_1=F_2=1=X_1$  and $F_4=3=X_2$ when $\delta=1+{\sqrt{2}}$.

Note that in the cases $d\in \{219,14401\}$, we only found one value of $\ell$ such that $X_{\ell}$ has Zeckendorf representation with at most two terms (instead of two such $\ell$), which is why these
$d$ are not included in the statement of the main result.




\section{Acknowledgements}


C. A. G. was supported in part by Project 71079 (Universidad del Valle). F. L. was supported by grant CPRR160325161141 and an A-rated scientist award both from the NRF of South Africa and by grant no. 17-02804S of the Czech Granting Agency.


\end{document}